\documentclass[draft]{amsart}

\usepackage{amsmath, amssymb, amsthm}
\usepackage{enumerate}
\usepackage{url}
\usepackage{amscd}

\theoremstyle{plain}

\newtheorem{theorem}{Theorem}[section]

\newtheorem{proposition}{Proposition}[section]

\newcommand{\R}{{\mathbb{R}}}

\numberwithin{equation}{section}

\begin{document}
\keywords{Integration over the Unit Sphere in $\R^n$, Integration over the Unit Ball in $\R^n$}
\subjclass[2020]{26B25, 42B99}
\title{Integration of Monomials over the Unit Sphere and Unit Ball in $\R^n$}

\author{Calixto P. Calder\'on}
\address{Dept of Math, Stat \& Comp Sci\\ University of Illinois at Chicago\\
 Chicago IL 60607} \email{calixtopcalderon@gmail.com}
\author{Alberto Torchinsky}
\address{Department of Mathematics\\ Indiana University\\
Bloomington IN 47405} \email{torchins@iu.edu}

\begin{abstract} We compute the integral   of monomials of the form $x^{2\beta}$   over the unit sphere and the unit ball  in $\R^n$  where $\beta=(\beta_1,\ldots,\beta_n)$  is a   multi--index  with real components  $\beta_k>-1/2$, $1\le k\le n$,  and discuss their asymptotic behavior as some, or all, $\beta_k\to  \infty$.   This allows for 
 the evaluation of    integrals involving  circular and hyperbolic trigonometric functions over the unit sphere   and the unit ball  in $\R^n$.
We also consider the Fourier transform of monomials  $x^{\alpha}$ restricted to the unit sphere  in $\R^n$, where the multi--indices $\alpha$ have  integer components,   and discuss their behaviour at the origin. 
\end{abstract}
\date{}

\maketitle

\section{Introduction}

Polynomial integration over the unit sphere is  used in a variety of applications   in 
Physics, including those situations that    involve integrands containing a Green’s function over the unit sphere. For example, 
Mura used a Fourier transform applied to an anisotropic Green’s function to integrate the displacement field inside an ellipsoidal domain over the unit
sphere \cite{TM}.  And,    Asaro and Barnett     integrated 
the strain field inside an ellipsoidal domain subjected to stress free strain (a polynomial of degree $M$)  over the
unit sphere   and observed that  the formulation becomes  complex for $ M > 1$,
 {\cite{AB}}. 

Othmani discussed these and other examples, including the  analytical expressions for the finite Eshelby tensors, and by means of  an inductive argument combined with the divergence theorem,   concluded that (in  dimensions $2$  and $3$) the integral  of monomials of order $k$ over the unit sphere   
vanishes if $k$ is odd, and satisfies an iterative relation if $k$ is even \cite{YO, YOHM}.  
There is a caveat here. By symmetry considerations, the integral over the unit sphere of   a monomial of even order  that contains an odd power of one of its space variables in its expression, also  vanishes. 

Higher dimensional results 
were considered by   Bochniak  and   Sitarz  when  addressing  the spectral interaction between universes  \cite{BS}.  And   Wang, when discussing the small sphere limit of quasilocal energy in higher dimensions along lightcone cuts,  raised the possibility of integrating the reciprocal of monomials over the unit sphere  \cite{JW}. 

From an analytical point of view  Baker and    Namazi independently  adopted a similar approach -- an inductive argument combined with the divergence theorem -- when considering  what Namazi called a generalized   Wallis formula \cite{JAB, N}. And, from a different perspective,  Kazarinoff  noted that  negative powers may be allowed in Wallis $1$-dimensional formula {\cite[p.19]{DK}.

As for our results, we begin by introducing some notations.  For an integer $k$, let $k{!!}=k(k-2)\cdots 2$ if $k$ is even, and $k{!!}=k(k-2)\cdots 1$ if $k$ is odd.  Given 
a multi-index $\alpha=(\alpha_1,\ldots,\alpha_n)$ of  nonnegative integers  and a scalar $\lambda$, let 
\[ \lambda\,\alpha=(\lambda\, \alpha_1,\ldots,\lambda\,\alpha_n),\quad\alpha{!}=\alpha_1{!}\cdots \alpha_n{!}\,,\quad |\alpha|=\alpha_1+\cdots +\alpha_n.
\]

And, for $ x=(x_1,\ldots,x_n)$  in $\R^n,$ let
$ x^\alpha = x_1^{\alpha_1} \cdots x_n^{\alpha_n}.$  
Now,   for 
a multi-index $\beta=(\beta_1,\ldots,\beta_n)$ of real numbers,    since  for  real numbers    $t,\delta$ we interpret 
$t^{2\delta}=\Big(\big( t^{2}\big)^{{1/2}}\Big)^{2 \delta}=|t|^{2 \delta},
$
  we have 
$ x^{2\beta } = x_1^{2\beta_1} \cdots x_n^{2\beta_n}= |x_1|^{2\beta_1} \cdots |x_n|^{2\beta_n}.
$

We will assume that  $n$ is greater than or equal
to $3$, and denote with  $d\sigma$  the element of surface area on the unit sphere $\partial B(0,1)$ in $\R^n$,  and with $\omega_n={2 \pi^{ n/2}}/{\Gamma(n/2)}$  the surface area of the unit sphere in $\R^n$. Here $\Gamma(s)$ denotes the Gamma function defined by $\Gamma(s)=\int_0^\infty e^{-t} t^{s-1}\,dt$, $s>0$.

Symmetry plays a role in our discussion. Indeed, the integral over the unit sphere of  any monomial $x^\alpha$ where $\alpha_j$ is odd for some $j$,   is $0$. And, since
\[\int_{\partial B(0,1)}x_j^2 \,d\sigma (x) =\int_{\partial B(0,1)}x_k^2 \,d\sigma (x)\,,\quad 1\le j,k\le n, \]  it readily follows that
\begin{equation*}\frac1{\omega_n} \int_{\partial B(0,1)}x_k^2 \,d\sigma (x)= \frac{1}{n}, \quad 1\le k\le n.
\end{equation*}
This expresses the most elementary form of a generalized   Wallis formula in higher dimensions.

In 1939 Hermann  Weyl  proved   what he called a well--known formula for calculating the mean value of the monomial $x^{2\alpha}$ over the unit sphere in $\R^n$,  {\cite[(12),  p.\! 465]{HW}}, to wit, 
\begin{equation}
\frac1{\omega_n}\, \int_{\partial B(0,1)}x^{2\alpha}\,d\sigma(x)=
\frac{\prod_{ \alpha_k\ne 0,k=1}^n (2\alpha_k-1){!!}}{ (n+2 |\alpha|-2)  \cdots (n+2) n}.
\end{equation}
 We refer to (1.1)  as   {\it{Wallis $n$--dimensional formula}}.

And, a similar result  holds for the unit ball in $\R^n$. Indeed, 
  passing to polar coordinates, since $v_n=\omega_n/n$, from (1.1) it follows that 
\begin{align}\frac1{v_n}\,  \int_{  B(0,1)}x^{2\alpha}\,d x &=n\,\int_0^1 r^{n+2|\alpha|-1} dr\, \frac1{\omega_n}\, \int_{\partial B(0,1)}x^{2\alpha}\,d\sigma(x)\nonumber
\\
&=
\frac{\prod_{ \alpha_k\ne 0,k=1}^n (2\alpha_k-1){!!}}{  (n+2 |\alpha|)  (n+2 |\alpha|-2)  \cdots (n+2) }.
\end{align}

 Weyl's result, and approach, was revisited in \cite{GBF}, where it is  remarked that Wallis formula remains valid for multi--indices $\beta=(\beta_1,\ldots,\beta_n)$ of non-negative real numbers $\beta_k$, $1\le k \le n$.

Wallis $n$--dimensional formula  is the underlying principle of this note, which is organized as follows. In Section 2 we show that
Weyl's approach extends to  monomials $x^{2\beta}$  where  the multi--indices $\beta$  have real components  $\beta_k>-1/2$, $1\le k\le n$, and discuss the asymptotic behaviour as some, or all, $\beta_k\to  \infty$. And,   similar results hold  for the unit ball in $\R^n$. 
In Section 3 we develop the tools to address the Fourier transform of monomials  $x^{\alpha}$
restricted to the unit sphere  in $\R^n$ when the multi--indices $\alpha$ have  integer components, which is then carried out in Section 4; these transforms, when evaluated at the origin, yield Wallis $n$--dimensional formula. And, in Section 5  we
apply Wallis $n$--dimensional formula  to  evaluate    integrals involving  circular and hyperbolic trigonometric functions over the unit sphere   and the unit ball  in $\R^n$. 
The Bernoulli and Euler numbers appear naturally in this context. 

\section{Generalized Wallis Formula in $\R^n$}

 {\it {Plane wave functions}}, that is,  functions   on $\R^n$  that are constant along the hyperplanes perpendicular to a fixed direction in $\R^n$, are an important tool when integrating over the unit sphere. 
 Fritz John  considered in particular  plane wave functions  $g(y \cdot x) $    of  $x\in \R^n$,  where  $g(s)$ is  a     continuous function  of the scalar variable $ s$ and $y$  a  fixed vector  in $\R^n$,
 which    are  constant along the hyperplanes perpendicular to the direction of $y$, 
and obtained the  fundamental identity 
\begin{equation*}\int_{\partial B(0,1)} g\big(y\cdot x\,\big) \,d\sigma(x)= \omega_{n-1}
\int_{-1}^1 \big(1-t^2)^{(n-3)/2}
g(|y|\,t)\,dt\end{equation*}
for the integral of  $g(x\cdot y)$ over the unit sphere {\cite[(1.2),  p.\! 8]{FJ}}.

In the particular case that  $g(t)=|t|^k$, where $k$ is an even integer,
    Fritz John  observed that  
\begin{equation*}\int_{\partial B(0,1)}\big|\, y\cdot x\,\big|^k \,d\sigma(x)= c_{k,n}\,|y|^k,
 \end{equation*}
and (1.1) may be derived  from this setting $k=2\,|\alpha|$.

The    formulation of Weyl's  result anticipated in the Introduction is the following:
\begin{proposition}
Let  $\beta=(\beta_1,\ldots, \beta_n)$ be a  multi-index  of real numbers with $\beta_k>-1/2$, $1\le k\le n$. Then,  the integral over the unit sphere in $\R^n$ of the monomial  $x^{2\beta}$  can be computed as
\begin{equation}\int_{\partial B(0,1)}x^{2 \beta} \,d\sigma (x) =2 \ \frac1{\Gamma(|\beta|+ n/2)}\, \prod_{k=1}^n \Gamma(\beta_k+1/2).
\end{equation}

In particular, when $\beta= \alpha$  is a multi--index of nonnegative integers, (1.1) holds.

Furthermore, asymptotically,
\begin{align}\int_{\partial B(0,1)} &x^{2 \beta} \,d\sigma (x)\nonumber\\
& \sim 2  (2 \pi)^{(m-1)/2} e^{(\beta_{m+1}+\cdots+\beta_n)} \prod_{k=m+1}^n \Gamma(\beta_k+1/2) \frac{\prod_{k=1}^m \beta_k^{\beta_k} }{ |\beta|^{|\beta|+(n-1)/2} }
\end{align}
as $\beta_1,\ldots, \beta_m\to\infty$, for $1\le m \le n$.

And,   if all the $\beta_j \to\infty$, we have
\[\int_{\partial B(0,1)} x^{2 \beta} \,d\sigma (x)\sim  2  (2 \pi)^{(n-1)/2} \frac{\prod_{k=1}^n \beta_k^{\beta_k} }{ |\beta|^{|\beta|+(n-1)/2} }.\]
\end{proposition}

\begin{proof}
Fix $\varepsilon>0$, let $B(0,\varepsilon)$ denote the ball of radius $\varepsilon$ centered at the origin,  and set
\[f(x,\varepsilon)= x^{2\beta}e^{-|x|^2} \big( 1-\chi_{B(0,\varepsilon)}(x)\big),\quad x\in \R^n.
\]

Then, passing to  polar coordinates, it readily follows that 
\begin{equation*}\int_{\R^n} f(x,\varepsilon)\,dx=\int_\varepsilon^\infty e^{-r^2} r^{2\,|\beta|+n-1}\, dr\,  \int_{\partial B(0,1)}x^{2 \beta} \,d\sigma (x),
\end{equation*}
where 
\[\int_\varepsilon^\infty e^{-r^2} r^{2\,|\beta|+n-1}\, dr =\frac12 \int_{\varepsilon^2}^\infty e^{-r} r^{|\beta|+n/2-1/2}r^{-1/2}\,dr.\]

Hence, since by assumption   $|\beta|=\beta_1+\cdots+\beta_n>-n/2$, it readily follows that 
\begin{equation}\lim_{\varepsilon\to 0} \int_{\R^n}  f(x,\varepsilon)\,dx =\frac12 \Gamma(|\beta|+ n/2)\,  \int_{\partial B(0,1)}x^{2 \beta} \,d\sigma (x).
\end{equation}

Let now $\chi_Q(0, 2\varepsilon)$ denote the characteristic function of the cube of sidelength $2\varepsilon$ centered at the origin, and observe that
\begin{align*}\int_{\R^n} f(x,\varepsilon)\,dx &=\int_{\R^n}  x^{2\beta}e^{-|x|^2} \big( 1-\chi_{Q(0,2\varepsilon)}(x)\big) \,dx\\
& +
\int_{\R^n}  x^{2\beta}e^{-|x|^2} \big(\chi_{Q(0,2\varepsilon)}(x)- \chi_{B(0,\varepsilon)}(x)\big)\,dx.
\end{align*}

Now, with $\alpha=1/\sqrt{2 n}$ it readily follows that
\[    \quad\Big|\chi_{Q(0,2\varepsilon)}(x)- \chi_{B(0,\varepsilon)}(x)\Big|\lesssim  \Big(\chi_{Q(0,2\varepsilon)}(x)- \chi_{Q(0,2\alpha \varepsilon)}(x)\Big) 
\]
and, consequently, 
that
\begin{align} \Big|\int_{\R^n}   x^{2\beta}e^{-|x|^2}&  \Big(\chi_{Q(0,2\varepsilon)}(x)- \chi_{Q(0,2\alpha \varepsilon)}(x)\Big) \,dx\Big|\nonumber\\
&\lesssim\sum_{k=1}^n  \int_{ 2\alpha \varepsilon}^{\varepsilon}  |x_k|^{2\beta_k}    \,dx_k
\lesssim\sum_{k=1}^n \varepsilon^{1+2\beta_k}, 
\end{align}
which, since $1+2\beta_k>0$ for all $k$, tends to $0$ as $\varepsilon\to 0$. 

Now, the integral 
\[
\int_{\R^n}  x^{2\beta}e^{-|x|^2} \big( 1-\chi_{Q(0,2\varepsilon)}(x)\big) \,dx\] 
can be computed as
\begin{align*}\prod_{k=1}^n \int_{\R}e^{-x_k^2}\,x_k^{2\beta_k} \big( 1-\chi_{Q(0,2\varepsilon)}(x_k)\big)dx_k 
&=2^n\prod_{k=1}^n \int_{\varepsilon}^\infty e^{-x_k^2}\,x_k^{2\beta_k}dx_k\\ 
&= \prod_{k=1}^n \int_{\varepsilon^2 }^\infty e^{-x_k}\,x_k^{\beta_k-1/2}dx_k  
\end{align*}
and, therefore,  we have
\begin{equation}\lim_{\varepsilon\to 0}\int_{\R^n}  x^{2\beta}e^{-|x|^2} \big( 1-\chi_{Q(0,2\varepsilon)}(x)\big) \,dx = \prod_{k=1}^n \Gamma(\beta_k+1/2)\
\end{equation}

Hence,  combining (2.4) and (2.5) we get
 \[\lim_{\varepsilon\to 0} \int_{\R^n}  f(x,\varepsilon)\,dx =\prod_{k=1}^n \Gamma(\beta_k+1/2),\]
which together with (2.3) gives 
\[
\frac12 \Gamma(|\beta|+ n/2)\,  \int_{\partial B(0,1)}x^{2 \beta} \,d\sigma (x)
=\prod_{k=1}^n \Gamma(\beta_k+1/2),\]
and (2.1) holds.

Now,  unraveling  the expression on the right--hand side of (2.1)  when $\beta=\alpha$ is an integer multi--index
it  follows  that in this case  also  (1.1) holds.

As for the asymptotic behavior of (2.1),  we recall Stirling's formula,
\[ \Gamma(x+a)\sim \sqrt{2 \pi x}\,  x^{x+a-1} e^{-x},
\]
where $a\ge 0$, as $x\to \infty$.

 Suppose that $\beta_1,\ldots,\beta_m\to\infty$, $1\le m\le n$.   Then we have
 \[\Gamma(\beta_k+1/2)\sim \sqrt{2 \pi }\,  \beta_k^{\beta_k} e^{-\beta_k},\quad  1\le k\le m,
\]
and
\[\Gamma(|\beta|+n/2)\sim \sqrt{2 \pi}\,  |\beta|^{\beta|+(n-1)/2} e^{-|\beta|},
\]
and, consequently,  
\begin{align} \frac1{\Gamma(|\beta|+ n/2)}\, \prod_{k=1}^m \Gamma(\beta_k+1/2)  &\sim \frac1{\sqrt{2 \pi}\,  |\beta|^{|\beta|+(n-1)/2} e^{-|\beta|}}\,  \prod_{k=1}^m  \sqrt{2 \pi }\,  \beta_k^{\beta_k} e^{-\beta_k}\nonumber
\\
& \sim (2 \pi)^{(m-1)/2}\,e^{(\beta_{m+1}+\cdots+\beta_n)} \, \frac{\prod_{k=1}^m \beta_k^{\beta_k} }{ |\beta|^{|\beta|+(n-1)/2} }.
\end{align}

(2.2) follows at once from (2.1) and (2.6). 
\end{proof}

As for  the integral over the unit ball in $ {\R^n} $ we have:
\begin{proposition}
Let  $\beta=(\beta_1,\ldots, \beta_n)$ be a  multi-index  of real numbers with $\beta_k>-1/2$, $1\le k\le n$. Then,  the integral of the monomial  $x^{2\beta}$ over the unit   ball in $\R^n$ can be computed as
\begin{equation}\int_{  B(0,1)}x^{2 \beta} \,d x  =  \ \frac1{\Gamma((|\beta|+1) + n/2)}\, \prod_{k=1}^n \Gamma(\beta_k+1/2).
\end{equation}

In particular, when $\beta= \alpha$  is a multi--index of nonnegative integers, (1.2) holds.

Furthermore, asymptotically,
\begin{align}\int_{ B(0,1)} &x^{2 \beta} \,d x \nonumber\\
& \sim  (2 \pi)^{(m-1)/2} e^{(\beta_{m+1}+\cdots+\beta_n)} \prod_{k=m+1}^n \Gamma(\beta_k+1/2) \frac{\prod_{k=1}^m \beta_k^{\beta_k} }{ |\beta|^{|\beta|+(n+1)/2} }
\end{align}
as $\beta_1,\ldots, \beta_m\to\infty$, for $1\le m \le n$.

And,   if all the $\beta_j \to\infty$, we have
\[\int_{  B(0,1)} x^{2 \beta} \,d x \sim     (2 \pi)^{(n-1)/2} \frac{\prod_{k=1}^n \beta_k^{\beta_k} }{ |\beta|^{|\beta|+(n+1)/2} }.
\]
\end{proposition}

\begin{proof}
For $\varepsilon>0$, consider 
\[g(x,\varepsilon)=x^{2\beta} \chi_{B(0,1)\setminus B(0,\varepsilon)(x)}.\]

Then, passing to polar coordinates it readily follows that
\begin{equation*}\int_{\R^n} g(x,\varepsilon)\,dx= \int_\varepsilon^1 r^{n+2|\beta |-1} dr\, \int_{\partial B(0,1)}x^{2\beta}\,d\sigma(x),
\end{equation*}
where, by assumption $n+2|\beta|>0$, and, consequently,
\[\lim_{\varepsilon\to 0} \int_\varepsilon^1 r^{n+2|\beta |-1} dr=\frac1{(2|\beta |+n )}. 
\]

Thus,  by (2.1) we have
\begin{align*}\int_{ B(0,1)}x^{2 \beta} \,dx &=\lim_{\varepsilon\to 0} \int_{\R^n} g(x,\varepsilon)\,dx \\
&=\frac1{(2|\beta |+n )} \,  \int_{\partial B(0,1)}x^{2\beta}\,d\sigma(x) \\
&=  \frac1{ \Gamma((|\beta|+1)+ n/2)}\, \prod_{k=1}^n \Gamma(\beta_k+1/2),
\end{align*}
which gives (2.7).

Moreover,  since
\[\frac{1}{2|\beta|+ n}\sim \frac{1}{2|\beta|} \quad {\text{as  $|\beta|\to\infty$,}} \]
 the  asymptotic values (2.8) follow at once from (2.2), and we have finished.
\end{proof}

\section{Bessel Functions}
In this section we cover the preliminary material to  discuss the Fourier transform of mononials restricted to the unit sphere in $\R^n$.

$J_\nu(x)$, the Bessel function of order $\nu$,  
 is defined as the solution of the second order linear equation
\[ t^2\, \frac{d^2y}{dt^2} + t \frac{dy}{dt}+(t^2-\nu^2)\,y=0. \]

Basic properties of the Bessel functions follow readily from their power series    expansion \cite {W},
\[ J_v(t)= \Big( \frac{t}{2}\Big)^\nu\sum_{k=0}^\infty \frac{(-1)^k}{k{!}\,\Gamma(\nu+k+1)}
\Big( \frac{t}{2}\Big)^{2t}.\]

They include:
\begin{equation*}\lim_{t\to 0^+}\frac{J_\nu(t)}{t^\nu}=2^{-\nu}\frac1{\Gamma(\nu+1)},
\end{equation*} 
 and the recurrence formula 
\begin{equation*}\frac{d}{dt}\Big(\frac{J_\nu(t)}{t^\nu}\Big)=-t^{-\nu}\, J_{\nu+1}(t).
\end{equation*}

In this note we will work with  the function $\Psi_\nu(t)$ defined by
\begin{equation}\Psi_\nu(t)=\frac{J_\nu(t)}{t^\nu}= 
2^{-\nu}\sum_{k=0}^\infty \frac{(-1)^k}{k{!}\,\Gamma(\nu+k+1)}
\Big( \frac{t}{2}\Big)^{2k}.
\end{equation}
 Then the above relations become, respectively,
\begin{equation}\lim_{t\to 0^+}\Psi_\nu(t)=2^{-\nu}\frac1{\Gamma(\nu+1)}\,,
\end{equation} 
 and 
\begin{equation}\Psi_\nu'(t)=-t\, \Psi_{\nu+1}(t).
\end{equation}

Now, repeated applications of (3.3) yield 
\begin{equation}  \Psi_\nu''(t) =- \Psi_{\nu+1}(t)+t^2\Psi_{\nu+2}(t),
\end{equation}
\begin{equation*}\Psi_\nu'''(t) =3t\Psi_{\nu+2}(t)-t^3\Psi_{\nu+3,}(t),
\end{equation*}
\begin{equation}  \Psi_\nu^{iv}(t) =3 \Psi_{\nu+2}(t)-6t^2\Psi_{\nu+3}(t)+t^4 \Psi_{\nu+4}(t),
\end{equation}
\[  \Psi_\nu^{v}(t) =-15 t\Psi_{\nu+3}(t)+10 t^3\Psi_{\nu+4}(t)-t^5 \Psi_{\nu+5}(t),
\]
and so on.

Thus,  the following pattern emerges. If $D$ denotes  differentiation with repect to $t$, 
 $D^k\Psi_\nu(t)$ is a polynomial  of degree $k$ in $t$ with coefficients $c_j\, \Psi_{\nu+j}(t)$ for some scalars $c_j$  with $0\le j\le k$. 
And, since as is readily seen from (3.1),  $\Psi_\nu(t)$ is an even function, its derivatives of odd order are odd functions, and those of even order are even.  

Moreover, for  integers $k=1,2,\ldots$, $D^{2k}\Psi_\nu(t)$ is an even  polynomial of degree $2k$ in $t$ consisting of $k+1$ terms that can be written as
\begin{align}D^{2k}&\Psi_\nu(t)\nonumber\\
&= t^{2k}\Psi_{\nu+2k}(t)- c_{2k-1}t^{2(k-1)}\Psi_{\nu+2k-1}(t)+\cdots+(-1)^k c_k\Psi_{\nu+k}(t).
\end{align}

In fact, it can be readily seen that if (3.6) holds for $2k$, then, on account of (3.4), it also holds for $2(k+1)$, and it is thus valid for all $2k$. 

A similar argument applies to $D^{2k+1}\Psi_\nu(t)$ for  integers $k=1,2,\ldots$ In this case $D^{2k+1}\Psi_\nu(t)$ is an odd polynomial of degree $2k+1$ in $t$ consisting of $k+1$ terms that can be written as
\begin{align}&D^{2k+1}\Psi_\nu(t)\nonumber\\
&= -t^{2k+1}\Psi_{\nu+2k+1}(t)+d_{2k}t^{2k-1}\Psi_{\nu+2k}(t)+\cdots+ (-1)^{k+1} d_k\,t \Psi_{\nu+k+1}(t),
\end{align}
and which can be  verified as in the even case.

We are   particularly   interested    in the constant term when the polynomial is even, and  in the term corresponding to $t$ when the polynomial  is odd. We adopt here the usual notation
that $\varphi(\xi)=o(|\xi|^\eta )$, where $\eta>0$,  as $|\xi|\to 0$, provided that 
\[ \lim_{|\xi|\to 0} \frac{|\varphi(\xi)|}{|\xi|^\eta}=0.
\]

We begin by proving:
\begin{proposition}
With $\nu>0$, let 
\[\Psi_\nu(\xi)=\frac{J_\nu(|\xi|)}{|\xi|^\nu}\,,\quad \xi\in \R^n.\] 
 Then, 
\begin{equation}\lim_{|\xi|\to 0}\Psi_\nu(\xi)=2^{-\nu}\frac1{\Gamma(\nu+1)}.
\end{equation}

Moreover, for $\varepsilon>0$,   $1\le j\le n$,  and $k=1,2,\ldots$,
\begin{equation}\frac{\partial^{2k}\Psi_\nu(\xi)}{\partial \xi_j^{2k}}=( -1)^k   (2k-1){!!} \, \Psi_{\nu+k}( \xi ) +o(|\xi|^{2-\varepsilon})
\end{equation}
as $|\xi|\to 0$.

And
\begin{equation}\frac{\partial^{2k+1}\Psi_\nu(\xi)}{\partial \xi_j^{2k+1}}=( -1)^{k+1} (2k+1){!!} \,  ( \xi_j)\,\Psi_{\nu+k+1}( \xi ) +o(|\xi|^{3-\varepsilon})
\end{equation}
as $|\xi|\to 0$.
\end{proposition}
\begin{proof}
(3.8) is a restatement of (3.2). Now, since 
\[\frac{\partial\Psi_\nu(\xi)}{\partial \xi_j}=\frac{\partial{\Big(  |\xi|^{-\nu}\,J_{\nu}( |\xi|)\Big) }}{{\partial \xi_j}}=
- \xi_j\, \frac{J_{\nu+1}( |\xi|)}{ |\xi|^{\nu+1}}=- \xi_j\,\Psi_{\nu+1}(\xi),
\]
and, similarly, 
\begin{equation*}\frac{\partial^{2}\Psi_\nu(\xi)}{\partial \xi_j^{2}}= -\Psi_{\nu+1}(\xi)+\xi_j^2\Psi_{\nu+2}(\xi),
\end{equation*}
 these expressions correspond to (3.3) and (3.4) with $\Psi_\nu(t)$ replaced by $\Psi_\nu(\xi)$ and  $t$ replaced by $\xi_j$ there.  The same applies to all other statements, so we will consider $\Psi_\nu(t)$ in what follows.

Now, in the even case we are interested in the  constant term. Setting  $t=0$ in  (3.6) it follows that
\begin{equation}D^{2k}\Psi_{\nu}(0)=  (-1)^k c_k \Psi_{\nu+k }(0).
\end{equation}

Also, by  (3.1) it readily follows that 
\begin{equation*}D^{2k} \Psi_\nu(0)=2^{-\nu} (-1)^k\frac1{k{!}}
\frac1{\Gamma(\nu+k+1) 2^{2k}}\,2k{!},
\end{equation*}
which, since 
\[\frac1{2^{2k}}\,\frac{2k{!}}{k{!}}=2^{-k} (2k-1){!!}
\]
can be restated as
\begin{equation*}D^{2k} \Psi_\nu(0)=(-1)^k 2^{-(\nu+k)}  
\frac1{\Gamma(\nu+k+1)}  (2k-1){!!},
\end{equation*}

Thus, since by (3.8)
\[ 2^{-(\nu+k)}\,\frac1{\Gamma(\nu+k+1)}=\Psi_{\nu+k }(0),
\]
it follows that 
\begin{equation}D^{2k} \Psi_\nu(0) =(-1)^k  (2k-1){!!} \,\Psi_{\nu+k}(0),
 \end{equation}
and by (3.11) and (3.12) we get
\[(-1)^kc_k \Psi_{\nu+k}(0)=(-1)^k (2k-1){!!}\,\Psi_{\nu+k}(0),
\]
and, therefore, $c_k=(2k-1){!!}$ 

We are also interested  in the coefficient of the term corresponding to $t$ when $k$ is odd. 
Note that, from (3.4) applied to (3.6) it follows that
\begin{equation}D^{2k+2} \Psi_\nu (t)=  (-1)^{k+1}2\, c_{k+1}\Psi_{\nu+k+1}(t)   + (-1)^{k+1} c_k \Psi_{\nu+k+1}(t)+ \Phi(t),
\end{equation}
where $\Phi(t)=o(|t|^{4-\varepsilon})$, and $\Phi(0)=0$.  Hence,
\begin{equation}D^{2k+2} \Psi_\nu (0)=  (-1)^{k+1}\big( 2c_{k+1}  +c_k\big) \Psi_{\nu+k+1}(0).
\end{equation}

And, since as in (3.12),  it follows that
\begin{equation*}
D^{2k+2} \Psi_\nu (0)=(-1)^{k+1} (2k+1){!!}\,\Psi_{\nu+k+1}(0),
\end{equation*}
 combining (3.13) and (3.14),  we get 
\begin{equation} 2 c_{k+1} +  c_k= (2k+1){!!},
\end{equation} 
which, since $c_k=(2k-1){!!}$, 
 implies  that
\begin{equation*} c_{k+1}=\frac12 \big((2k+1){!!}-(2k-1){!!}\big)=
k  \, (2k-1){!!} 
\end{equation*}

To determine $d_k$ now, observe that differentiating  (3.5), from (3.3) it readily follows that
\[D^{2k+1}\Psi_\nu(t)= (-1)^{k+1} ( c_k+2c_{k+1}) t\, \Psi_{\nu+k+1}(t)+ \Phi(t),
\]
where $\Phi(0)=0$, and, therefore, comparing this expression with (3.6) it readily follows that
$d_k=c_k+2c_{k+1}.$
Hence, from (3.15)   we conclude that
\[  d_k= (2k+1){!!}
\]
and the proof is finished.
\end{proof}

The  above reasoning allows for the determination of all the desired coefficients in (3.6) and (3.7). 
For example, one  may verify that  $c_{k+2}$ in (3.6) is equal to
\[c_{k+2}= 
\frac1{3{!}}\, k(k-1)   (2k-1){!!}    
\]

At this time, no obvious pattern  to compute the various coefficients is apparent to us.  

\section{The Fourier Transform of Monomials Restricted to  the Unit Sphere}
The theory of  polynomials on the unit sphere of $\R^n$  is developed in 
{\cite[Chapter 5]{ABR}}, where in particular a refined version of Weyl's result (1.1) for harmonic polynomials is given.  Here we complement those results by  computing the Fourier transform of  monomials restricted to the unit sphere of $\R^n$ and discussing their behaviour at the origin.. 

Recall that  when  $n\ge 3$,   the Fourier transform of the surface area measure carried on the
sphere $\partial  B(0, 1)$  centered at the origin of radius 1 in $\R^n$, 
is given by  {\cite[p.\! 154]{SW}},   
 \begin{equation} \int_{\partial B(0,1)}  e^{-ix\cdot\xi}\,d\sigma (x) 
=(2\pi)^{n/2}\,
 \Psi_{(n-2)/2}(  \xi ).
 \end{equation}

To fix ideas we consider  first the Fourier transform of the monomial $x_1^{2k}$ restricted to the unit sphere. 
 Differentiating (4.1) with respect to $\xi_1$ it follows that
\begin{align*}\frac{1}{(2\pi)^{n/2}}\, \int_{\partial B(0,1)} & x_1^{2k}\,e^{-ix\cdot \xi}\,d\sigma (x)    
\\
&=  (2k-1){!!}\,\,  \Psi_{(n/2)+k-1}(\xi) - k(2k-1){!!}\, \xi_1^2\,\Psi_{(n/2)+k}(\xi)
\\
&+\frac1{3{!}}k(k-1) (2k-1){!!}\,\xi_1^4\,  \Psi_{(n/2)+k+1}(\xi)+ o(|\xi|),
\end{align*}
and we are done in this case. 

For the applications we have in mind it suffices to invoke  Proposition 3.1 with $\nu=(n-2)/2$ there,
and observe that 
\begin{equation}\int_{\partial B(0,1)}x_1 ^{2k} e^{-ix\cdot \xi}   \,d\sigma (x)=(2\pi)^{n/2}\,  (2k-1){!!} \,  \Psi_{(n/2)+k-1}(\xi) +o(|\xi|).
\end{equation}

It is important to note that all the summands that appear in the $o(|\xi|)$ term   have  a    positive power of   $\xi_1$ in them. The reason  being that, when derivatives are taken with respect to variables other than $\xi_1$, the expression remains  an  $o(|\xi|)$ term.

Next, concerning the monomial $x_1^{2k+1}$, 
differentiating (4.2) one more time with respect to $\xi_1$, by (3.10) in  Proposition 3.1 with $\nu=(n-2)/2$ there 
it follows that  
\begin{align}\int_{\partial B(0,1)}x_1 ^{2k+1}&e^{-ix\cdot \xi}   \,d\sigma (x)\nonumber\\
&=(2\pi)^{n/2}   (2k+1){!!} \,(i \,\xi_1 )\,  \Psi_{(n/2)+k}(\xi) +o(|\xi|),
\end{align}
which gives the desired result in this case.

 We introduce some notations to address the general case. 
Given a multi-index $ \alpha=(\alpha_1,\ldots, \alpha_n)$,  let  $ \alpha^o=(\alpha_1^o,\ldots, \alpha_n^o)$ where  $\alpha_j^o=\alpha_j$  if $\alpha_j$ is an  odd integer and $=0$ otherwise,  $\mu(\alpha^o)$ the number of indices with  $\alpha_j^o\ne 0$,  and $ \alpha^e=(\alpha_1^e,\ldots, \alpha_n^e)$,  where $\alpha_j^e=\alpha_j-\alpha_j^o$. 

Let 
\[ A_\alpha(\xi)=(2\pi)^{n/2}\,\Big(  \prod_{ \alpha_m^e\ne 0,m=1}^n (\alpha_m^e-1){!!}\Big)\Big(\prod_{ \alpha_m^o\ne 0,m=1}^n    \alpha_m^o {!!}\,  ( i \xi_m) \Big).
\]

We   consider now the general case,  which follows essentially by iterating (4.2) and (4.3):

\begin{theorem}
Let $n\ge 3$. Then, for a  multi-index $\alpha$ we have  
\begin{align} 
\int_{\partial B(0,1)}x^{\alpha}& \,e^{-ix\cdot \xi}\, d\sigma (x)\nonumber\\
& = 
A_\alpha(\xi)\,\Psi_{(n +|\alpha|+\mu(o)-2)/2 }(\xi) +o(|\xi|).
\end{align}
\end{theorem}
\begin{proof}

 Given a multi-index $\alpha$, let $\alpha=\alpha^e +\alpha^o$.
Then,  iterating the relation (4.3) for those  $\xi_m$ with $\alpha_m^o\ne 0$ it readily follows that
\begin{align}\int_{\partial B(0,1)} & x^{\alpha^o} e^{-ix\cdot \xi}   \,d\sigma (x)
\nonumber\\
&=(2\pi)^{n/2}\, \Big(\,\prod_{\alpha_m^o\ne 0,m=1}^n    (\alpha_m^o){!!}\,(i\xi_m)\,\Big)\,   \Psi_{(n+|\alpha^o|+\mu(o)-2)/2}(\xi)+ o(|\xi|).
\end{align}

Now,   iterating the relation (4.2) for those  $\xi_m$ with $\alpha_m^e\ne 0$, since
$(n+|\alpha^o|+\mu(o)-2)/2+|\alpha^e|/2=(n+|\alpha|+\mu(o)-2)/2$, from (4.5)  it readily follows that 
\begin{align*}\int_{\partial B(0,1)}  x^{\alpha } e^{-ix\cdot \xi}   \,d\sigma (x)
&=\int_{\partial B(0,1)} x^{\alpha^o}  x^{\alpha^e} e^{-ix\cdot \xi}   \,d\sigma (x)\\
&=A_\alpha(\xi) \,   \Psi_{(n+|\alpha|+\mu(o)-2)/2}(\xi)+ o(|\xi|).
\end{align*}
and we have finished. 
\end{proof}

Wallis   $n$--dimensional  formula follows letting $\xi\to 0$ in (4.4). Indeed, the limit is equal to
\[ A_\alpha(0)\,\Psi_{(n +|\alpha|+\mu(o)-2)/2 }(0),
\]
which in turn is $0$ if $\mu(o)\ne 0$. Thus, we may assume that all the $\alpha_m$ are even, and then on account of (3.8) we have
that the limit is equal to
\begin{align*}(2\pi)^{n/2}&\,\Big(  \prod_{ \alpha_m\ne 0,m=1}^n (\alpha_m-1){!!}\Big)
\,\Psi_{(n +|\alpha|-2)/2 }(0)
\\
&=(2\pi)^{n/2} 2^{-(n +|\alpha|-2)/2}\frac1{\Gamma((n +|\alpha|)/2)}\,\Big(  \prod_{ \alpha_m\ne 0,m=1}^n (\alpha_m-1){!!}\Big),
\end{align*}
and the conclusion follows by a computational argument left to the reader.

\section{Applications}
We close the note with the integration of  circular and hyperbolic trigonometric functions over the unit sphere and the unit ball in $\R^n$. The results follow from the known fact that a power series converges uniformly
on compact subsets of its disc of convergence, and, therefore,  it   can
be integrated  termwise there. As for the uniform convergence of the series,  one may invoke the known fact that if $f(t)$  is analytic in
the disc $|t| < R$ , it is the sum of its Taylor series there, or else establish the convergence 
 invoking the asymptotics for the Bernoulli numbers, $B_k$,  or Euler numbers, $E_k$, \cite {TA, BE}.

Recall that 
\begin{equation*} t \coth(t)= \sum_{k=0}^\infty    \frac{2^{2k}}{(2k){!}}\,B_{2k}\, t^{2k}, 
\end{equation*}
where the series, by either criteria described above,  converges uniformly and absolutely for  $|t|<\pi$, 
and so  for a multi-index $\alpha$,   since  $B_0=1$,  the series  
\begin{equation}
x^\alpha\coth(x^\alpha)=1+
\sum_{k=1}^\infty    \frac{2^{2k}  }{(2k){!}}\,B_{2k} \,x^{2k\alpha }
\end{equation}
  converges uniformly and absolutely for  $|x|<R$ for some $R>1$. We may then integrate (5.1) termwise, and
 by (1.1)  or (2.1)  
 conclude that
\begin{align*}\frac1{\omega_n}\int_{\partial B(0,1)} x^\alpha\,&\coth(x^\alpha) \,d\sigma (x)\\
&= 1+\sum_{k=1}^\infty \frac{2^{2k} B_{2k} }{(2k){!}}\frac{\prod_{  \alpha_m\ne 0,m=1}^n (2k\alpha_m-1){!!}}{
(n+ 2|\alpha|\,k-2)\cdots (n+2)n}.
\end{align*}

Similarly,  for the integral over the unit ball,  by (1.2) or (2.7) 
 it follows that 
\begin{align*}\frac1{v_n}\int_{ B(0,1)} &x^\alpha \coth(x^\alpha) \,dx\\
&=1+\sum_{k=1}^\infty \frac{2^{2k} B_{2k} }{(2k){!}}\,\frac{\prod_{ \alpha_m\ne 0,m=1}^n (2k\alpha_m-1){!!}}{ (n+2|\alpha|k) (n+2|\alpha|k-2)  \cdots (n+2) }.
\end{align*}

As for circular trigonometric functions, consider, for instance,  the series expansion  for $t/\sin(t)$ given by 
\begin{equation*}\frac{ t}{\sin(t)}   = 1+ 2\sum_{k=1}^\infty (-1)^{k -1} \frac{(2^{2k-1}-1)  }{(2k){!}}\,B_{ 2k}\,t^{2k },  
\end{equation*}
which  converges uniformly and absolutely for   $|t|< \pi$. Hence,
for a multi-index $\alpha$   the series
\begin{equation}\frac{ x^\alpha}{\sin(x^\alpha)}   
= 1+ 2\sum_{k=1}^\infty   (-1)^{k -1} \frac{(2^{2k-1}-1)  }{(2k){!}}\,B_{ 2k}\,x^{2
k \alpha}, 
\end{equation}
converges uniformly and absolutely for $|x|<R$ for some $R>1$, and  we may evaluate the integral of (5.2) over the unit sphere and the unit ball as above. 

A slight variant works when   odd powers  are involved   in the expansion. For instance, consider  the tangent, which can be expanded as 
\begin{equation*}\tan(t)=\sum_{k=1}^\infty \tau_{2k-1} \,   B_{2k}\, t^{2k-1},
\end{equation*}
where
\[\tau_{2k-1}=(-1)^k 2^{2k} 
 \frac{(2^{2k-1}-1)}{(2k){!}}, \quad k=1,2,\ldots  \]
which 
  converges absolutely and uniformly for $|t|<  \pi/2 $.

Thus,  for a multi-index $\alpha$   the series
\begin{equation}\tan(x^\alpha) =
 \sum_{k=1}^\infty\tau_{2k-1} \, B_{2k}\,   x^{(2k -1)\alpha}
\end{equation}
converges absolutely and uniformly for  $|x|< R$, for some $R>1$, and therefore (5.3) 
 may be integrated  termwise over the unit sphere and the unit ball in $\R^n$.

Now,  the  monomials in (5.3)  integrate  to $0$ unless all the coordinates of  the   $ (2k  -1)\,\alpha $   are even integers, which is the case when all the $\alpha_j$ are even.
Let then $\alpha=2\alpha' $,
where the $\alpha_j'$ are nonnegative integers. 

Then, on account of (1.1) for $k=1,2,\ldots$,   
we have
\begin{align*}\frac1{\omega_n}\,\int_{\partial B(0,1)}  x^{ (2k-1)\alpha}  \,d\sigma (x) &
=\frac1{\omega_n}\,\int_{\partial B(0,1)}  x^{2 (2k-1)\alpha'}  \,d\sigma (x)\\ 
&=
 \frac{\prod_{  \alpha_m'\ne 0,m=1}^n (2(2k-1)\alpha_m'-1){!!}}{
(n+2(2k-1) |\alpha'| -2)\cdots (n+2) n }\\
&=
\frac{\prod_{  \alpha_m\ne 0,m=1}^n  ( (2k-1)\alpha_m-1){!!}}{
(n+(2k-1)|\alpha|-2)\cdots (n+2) n },
\end{align*}
and the answer follows readily from this by integrating (5.3) termwise.

As for the integral over the unit ball, observe that, similarly,  by (1.2), 
\begin{align*}\frac1{v_n}\,\int_{ B(0,1)} & x^{(2k-1)\alpha}  \,dx\\
&=  \frac{\prod_{  \alpha_m\ne 0,m=1}^n ((2k-1)\alpha_m-1){!!}}{(n+(2k-1)|\alpha| )
(n+(2k-1)|\alpha| -2)\cdots (n+2) },
\end{align*}
and the integral can be readily obtained by integrating (5.3) termwise.

Which brings us to the closing  remarks.  G. S. Ely considered  the question of expanding powers of trigonometric functions
given the expansions of the functions themselves \cite{GE}. In the case of the tangent 
one derives
\begin{equation*}\tan^3(t) = \sum_{k=1}^\infty\big(\, k  \, (2k+1)\, \tau_{2k+1}\,   B_{2k+2}-\tau_{2k-1} \,   B_{2k}\, \big) \, t^{2k-1}.
\end{equation*}

We leave to the reader the evaluation of the integral
 \[\int_{\partial B(0,1)}\tan^3(x^{\alpha})   \,d\sigma (x)\,.\]

Finally, unlike the functions considered above,
the expansion of the secant 
involves the {\it {Euler numbers}}, $ E_k$, and is given by
\[\sec(t)=\sum_{k=0}^\infty(-1)^k\,  \frac1{(2k){!}} E_{2k} \, t^{2k}.
\]

 Simple algebraic manipulations give that   the expansion for  $\sec^3(t)$ is given by
\[\sec^3(t)=\frac12\sum_{k=0}^\infty (-1)^k\frac1{(2k){!}} \big( E_{2k} - E_{2(k+1)}\big)t^{2k}.
\]

Then, for an appropriate  real multi-index $\beta$   the reader is invited to evaluate
\[  \int_{  B(0,1)}  \sec^3(x^{\beta})  \,dx.
\]

ORCID 

Calixto P. Calder\'on  https://orcid.org/0000-0002-4211-2110

Alberto Torchinsky https://orcid.org/0000-0001-8325-3617

\end{document}